\documentclass[12pt,final]{article}
\usepackage[margin=2.5cm,top=1.5cm]{geometry}
\usepackage[all,arc]{xy}
\usepackage{mathrsfs}
\usepackage{tikz}
\usepackage{graphicx}
\usepackage{mathtools}
 \usepackage{subfigure}
\usepackage[centerlast]{caption2}
\usepackage{amsmath,amsthm,amssymb,enumerate}
\usepackage{color}
\usepackage{setspace}
\newtheorem{them}{Theorem}[section]
\newtheorem{defn}{Definition}[section]
\newtheorem{lem}[defn]{Lemma}

\newtheorem{cor}[defn]{Corollary}

\newtheorem{prop}[defn]{Proposition}
\newtheorem{error}{Error}[section]
\newtheorem{exam}[error]{Example}
\newtheorem{con}[defn]{Conjecture}

\newtheorem{rem}[defn]{Remark}

\numberwithin{equation}{section}
\newcommand\keywordsname{Key words}
\newcommand\AMSname{AMS subject classifications}

\newenvironment{@abssec}[1]{%
     \if@twocolumn
       \section*{#1}%
     \else
       \vspace{.05in}\footnotesize
       \parindent .2in
         {\upshape\bfseries #1. }\ignorespaces
     \fi}
     {\if@twocolumn\else\par\vspace{.1in}\fi}

\begin{document}

\title{Spectrum of a class of matrices and its applications\footnote{L. You's research is supported by the Zhujiang Technology New Star Foundation of
Guangzhou (Grant No. 2011J2200090) and Program on International Cooperation and Innovation, Department of Education,
Guangdong Province (Grant No.2012gjhz0007).}}

\author{ Lihua You
\footnote{{\it{Corresponding author:\;}}ylhua@scnu.edu.cn.} %; {\it{phone number:\;}}+8613660067717.}
 \qquad Man Yang \footnote{{\it{Email address:\;}}784623899@qq.com. }
 %; {\it{phone number:\;}}+8618602062229.}
\qquad Jinxi Li\footnote{{\it{Email address:\;}}784623899@qq.com. }
\qquad Liyong Ren\footnote{{\it{Email address:\;}}275764430@qq.com. }
 }\vskip.2cm
\date{{\small
School of Mathematical Sciences, South China Normal University,\\
Guangzhou, 510631, P.R. China\\
}}
\maketitle

\noindent {\bf Abstract }
In this paper, we give the spectrum of a matrix by using the quotient matrix, then we apply this result to various matrices associated
 to a graph and a digraph, including adjacency matrix, (signless) Laplacian matrix, distance matrix, distance (signless) Laplacian matrix,
to obtain some known and new results. Moreover, we propose some problems for further research.



\section{Introduction}

\hskip.5cm
We begin by recalling some definitions.
Let $M$ be an $n\times n$  matrix, $\lambda_1, \lambda_2, \ldots, \lambda_n$ be the eigenvalues of $M$.
 It is obvious that the eigenvalues may be complex numbers since $M$ is not symmetric in general.
 We usually assume that $|\lambda_1|\geq |\lambda_2|\geq\ldots \geq |\lambda_n|$.
 The spectral radius of $M$ is defined as $\rho(M)=|\lambda_1|$, i.e., it is the largest modulus
 of the eigenvalues of $M$.   If $M$ is a nonnegative matrix, it follows from the Perron-Frobenius theorem that
 %$\lambda_1$ is a nonnegative number and
 the spectral radius $\rho(M)$ is a eigenvalue of $M$.
 If $M$ is a nonnegative irreducible  matrix, it follows from the Perron-Frobenius theorem that
 $\rho(M)=\lambda_1$ is simple.

Let $G$ be a connected graph with vertex set $V(G)$ and edge set $E(G)$. Let $A(G)=(a_{ij})$ denote the adjacency matrix of $G$, where $a_{ij}$ is equal to the number of edges $v_iv_j$. The spectral radius of $A(G)$, denoted by $\rho(G)$, is called the spectral radius of $G$. Let $diag(G)=diag(d_1, d_2, \ldots, d_n)$ be the diagonal matrix with degree of the vertices of $G$ and $Q(G)=diag(G)+A(G)$ be the signless Laplacian matrix of $G$, $L(G)=diag(G)-A(G)$ be the Laplacian matrix of $G$.
 The spectral radius of $Q(G)$, denoted by $q(G)$, is called the signless Laplacian spectral radius of $G$. The spectral radius of $L(G)$, denoted by $\mu(G)$, is called the Laplacian spectral radius of $G$.

For $u,v\in V(G)$, the distance between $u$ and $v$, denoted by $d_G(u,v)$ or $d_{uv}$,
is the length of the shortest path connecting them in $G$.
For $u\in V(G)$, the transmission of vertex $u$ in $G$ is the sum of distances  between $u$ and
all other vertices of $G$, denoted by $Tr_G(u)$.

Let $G$ be a connected graph with vertex set $V(G)=\{v_1, v_2, \ldots, v_n\}$.
The distance matrix of $G$ is the $n\times n$ matrix $\mathcal{D}(G)=(d_{ij})$ where $d_{ij}=d_{v_iv_j}$.
The distance  spectral radius of $G$, denoted by $\rho^{\mathcal{D}}(G)$, is the spectral radius of $\mathcal{D}(G)$,
 which is the largest eigenvalue of $\mathcal{D}(G)$.

In fact, for $1\leq i\leq n$, the transmission of vertex $v_i$, $Tr_G(v_i)$ is just the $i$-th row sum
 of $\mathcal{D}(G)$. Let $Tr(G)=diag(Tr_G(v_1), Tr_G(v_2), \ldots, Tr_G(v_n))$ be the diagonal matrix of vertex transmission of $G$. M. Aouchiche and P. Hansen \cite{2013LAA7} introduced the Laplacian and the signless Laplacian for the distance matrix of a connected graph. The matrix $\mathcal{L}(G)=Tr(G)-\mathcal{D}(G)$ is called the distance Laplacian of $G$, while the matrix $\mathcal{Q}(G)=Tr(G)+\mathcal{D}(G)$ is called the distance signless Laplacian matrix of $G$.
 \noindent It is obvious that $\mathcal{Q}(G)$ is irreducible, nonnegative, symmetric and positive semidefinite.
 The distance signless Laplacian spectral radius of $G$, denoted by $q^{\mathcal{D}}(G)$,  is the spectral radius of $\mathcal{Q}(G)$,
 which is the largest eigenvalue of $\mathcal{Q}(G)$.  The spectral radius of $\mathcal{L}(G)$, denoted by $\mu^{\mathcal{D}}(G)$, is called the distance Laplacian spectral radius of $G$.

Since $G$ is a connected graph, then $A(G)$, $Q(G)$, $\mathcal{D}(G)$ and  $\mathcal{Q}(G)$ are nonnegative  irreducible matrices, it follows  from the Perron Frobenius Theorem that
  $\rho(G)$,  $q(G)$,  $\rho^{\mathcal{D}}(G)$ and $q^{\mathcal{D}}(G)$
   are real numbers and there is a positive unit eigenvector corresponding to $\rho(G)$,  $q(G)$,  $\rho^{\mathcal{D}}(G)$ and $q^{\mathcal{D}}(G)$, respectively.

Let $\overrightarrow{G}=(V(\overrightarrow{G}), E(\overrightarrow{G}))$ be a digraph,
where $V(\overrightarrow{G})=\{v_1, v_2,
\ldots, v_n \}$ and $E(\overrightarrow{G})$ are the  vertex set and arc set of $\overrightarrow{G}$,
respectively. A digraph $\overrightarrow{G}$ is simple if it has no loops and
multiple arcs. A digraph  $\overrightarrow{G}$ is strongly connected if for every
pair of vertices $v_i, v_j\in V(\overrightarrow{G})$, there are  directed paths from $v_i$
to $v_j$ and from $v_j$ to $v_i$. In this paper, we consider finite, simple strongly connected digraphs.
%We follow \cite{1976, 1979, 1986, 1988, 2001} for terminology and notations.

Let $\overrightarrow{G}$ be a digraph. If two vertices are connected by an arc, then they are called adjacent.
For $e=(v_i, v_j)\in E(\overrightarrow{G})$, $v_i$ is the tail (the initial vertex) of $e$,
$v_j$ is the head (the terminal vertex)  of $e$. %or vertex $v_i$ is a tail of vertex $v_j$ and $v_j$ is a head of $v_i$.

Let $N^-_{\overrightarrow{G}}(v_i)=\{v_j\in V(\overrightarrow{G})|(v_j, v_i)\in E(\overrightarrow{G})\}$ and
$N^+_{\overrightarrow{G}}(v_i)=\{v_j\in V(\overrightarrow{G})|$ $(v_i, v_j) \in E(\overrightarrow{G})\}$
denote the in-neighbors and out-neighbors of $v_i$, respectively.

%Let $\overrightarrow{P_n}$ and $\overrightarrow{C_n}$ denote the directed path and the directed cycle on $n$ vertices, respectively.
%Let $\overset{\longleftrightarrow}{K_n}$ denote the complete digraph on $n$ vertices in which two arbitrary vertices
%$v_i, v_j\in V(\overset{\longleftrightarrow}{K_n})$, there are arcs $(v_i, v_j), (v_j, v_i)\in E(\overset{\longleftrightarrow}{K_n})$.

 For a digraph $\overrightarrow{G}$, let $A(\overrightarrow{G})=(a_{ij})$ denote the adjacency matrix of $\overrightarrow{G}$,
 where $a_{ij}$ is equal to the number of arcs $(v_i, v_j)$.
 The spectral radius of $A(\overrightarrow{G})$, denoted by $\rho(\overrightarrow{G})$, is called the spectral radius of $\overrightarrow{G}$.

 Let $diag(\overrightarrow{G})=diag(d^+_1, d^+_2, \ldots, d^+_n)$ be the diagonal matrix with outdegree of the vertices of  $\overrightarrow{G}$ and $Q(\overrightarrow{G})= diag(\overrightarrow{G})+A(\overrightarrow{G})$
 be the signless Laplacian matrix of $\overrightarrow{G}$, $L(\overrightarrow{G})=diag(\overrightarrow{G})-A(\overrightarrow{G})$ be the Laplacian matrix of $\overrightarrow{G}$.
 The spectral radius of $Q(\overrightarrow{G})$,  $\rho(Q(\overrightarrow{G}))$, denoted by $q(\overrightarrow{G})$,
 is called the signless Laplacian spectral radius of $\overrightarrow{G}$.
% Let $P_{\overrightarrow{G}}(x)$ denote the signless Laplacian characteristic polynomial of $Q(\overrightarrow{G})$.

For $u,v\in V(G)$, the distance from $u$ to $v$, denoted by $d_{\overrightarrow{G}}(u,v)$ or $d_{uv}$,
is the length of the shortest directed path from $u$ to $v$ in ${\overrightarrow{G}}$.
For $u\in V({\overrightarrow{G}})$, the transmission of vertex $u$ in ${\overrightarrow{G}}$ is the sum of distances
 from $u$ to all other vertices of ${\overrightarrow{G}}$, denoted by $Tr_{{\overrightarrow{G}}}(u)$.

Let ${\overrightarrow{G}}$ be a connected digraph with vertex set $V({\overrightarrow{G}})=\{v_1, v_2, \ldots, v_n\}$.
The distance matrix of ${\overrightarrow{G}}$ is the $n\times n$ matrix $\mathcal{D}({\overrightarrow{G}})=(d_{ij})$
where $d_{ij}=d_{\overrightarrow{G}}(v_i,v_j)$.
The distance  spectral radius of $\overrightarrow{G}$, denoted by $\rho^{\mathcal{D}}(\overrightarrow{G})$,
is the spectral radius of $\mathcal{D}(\overrightarrow{G})$.
%which is the largest eigenvalue of $\mathcal{D}(\overrightarrow{G})$.

In fact, for $1\leq i\leq n$, the transmission of vertex $v_i$, $Tr_{\overrightarrow{G}}(v_i)$ is just the $i$-th row sum
 of $\mathcal{D}(\overrightarrow{G})$.
 Let $Tr(\overrightarrow{G})=diag(Tr_{\overrightarrow{G}}(v_1), Tr_{\overrightarrow{G}}(v_2), \ldots, Tr_{\overrightarrow{G}}(v_n))$ be the diagonal matrix of vertex transmission of $\overrightarrow{G}$.
 The distance signless Laplacian matrix of ${\overrightarrow{G}}$ is the $n\times n$ matrix defined similar to the undirected graph
 by Aouchiche and Hansen as

 $$\mathcal{Q}({\overrightarrow{G}})=Tr({\overrightarrow{G}})+\mathcal{D}({\overrightarrow{G}}).$$

 \noindent Let $\mathcal{L}(\overrightarrow{G})=Tr({\overrightarrow{G}})-\mathcal{D}({\overrightarrow{G}})$ be the distance Laplacian matrix of $\overrightarrow{G}$.
 The distance signless Laplacian spectral radius of $\overrightarrow{G}$, $\rho(\mathcal{Q}(\overrightarrow{G}))$,
 denoted by $q^{\mathcal{D}}(\overrightarrow{G})$,  is the spectral radius of $\mathcal{Q}(\overrightarrow{G})$.
 %which is the largest eigenvalue of $\mathcal{Q}(G)$.

Since $\overrightarrow{G}$ is a simple strongly connected digraph,
then $A(\overrightarrow{G})$, $Q(\overrightarrow{G})$, $\mathcal{D}({\overrightarrow{G}})$ and  $\mathcal{Q}({\overrightarrow{G}})$ are nonnegative irreducible matrices.
It follows from the Perron Frobenius Theorem that
  $\rho(\overrightarrow{G})$, $\rho(Q(\overrightarrow{G}))=q(\overrightarrow{G})$, $\rho^{\mathcal{D}}(\overrightarrow{G})$ and $\rho(\mathcal{Q}(\overrightarrow{G}))=q^{\mathcal{D}}(\overrightarrow{G})$
  are positive real numbers and there is a positive unit eigenvector corresponding to $\rho(\overrightarrow{G})$, $q(\overrightarrow{G})$, $\rho^{\mathcal{D}}(\overrightarrow{G})$ and $q^{\mathcal{D}}(\overrightarrow{G})$, respectively.

For a connected graph $G=(V(G), E(G))$, the vertex connectivity of a graph denoted by $\kappa(G)$, is the minimum number of vertices whose deletion yields the resulting graph disconnected. Clearly, let $G$ be a connected graph on $n$ vertices, then $1\leq \kappa(G)\leq n-1$.
Similarly, for a strongly connected digraph $\overrightarrow{G}=(V(\overrightarrow{G}), E(\overrightarrow{G}))$, the vertex connectivity of a digraph denoted by $\kappa(\overrightarrow{G})$, is the minimum number of vertices whose deletion yields the resulting digraph non-strongly connected. Clearly, let $\overrightarrow{G}$ be  a strongly connected digraph with $n$ vertices, then $1\leq \kappa(\overrightarrow{G})\leq n-1$.

There are many literatures about graphs' and digraphs' connectivity. For early work, see \cite{2009LAA}, Ye-Fan-Liang characterize the graphs with the minimal least eigevalue among all  graphs with given vertex connectivity or edge connectivity. In 2010, Ye-Fan-Wang \cite{2010LAA} characterize the graphs with maximum signless Laplacian or adjacency spectral radius among all graphs with fixed order and given vertex or edge connectivity. Liu \cite{2010b} characterized the minimal distance spectral radius of  simple connected graphs with given vertex connectivity, or matching number, or chromatic number, respectively. Brualdi \cite{2010LAA1} wrote a stimulating survey on this topic.

In 2012, Lin-Shu-Wu-Yu \cite{2012DM} establish some upper or lower bounds for digraphs with some given graph parameters, such as clique number, girth, and vertex connectivity, and characterize the corresponding extremal graphs, give the exact value of the spectral radii of those digraphs. Besides, Lin-Yang-Zhang-Shu \cite{2012DM2} characterize the extremal digraphs (graphs) with minimum distance spectral radius among among all digraphs (graphs) with given vertex (edge) connectivity.

In 2013, Lin-Drury \cite{2013DM} characterize the extremal digraphs which attain the maximum Perron root of digraphs with given arc connectivity and number of vertices. Lin-Shu \cite{2013DAM} determine the extremal digraph with the minimal distance spectral radius with given arc connectivity. Xing-Zhou \cite{2013} determine the graphs with minimal distance signless Laplacian spectral radius among the connected graphs with fixed number of vertices and connectivity. Oscar Rojo and Eber Lenes \cite{2013LAA3} obtained a sharp upper bound on the incidence energy of graphs in terms of connectivity. Furthermore, some upper or lower bounds were obtained by the outdegrees and the average 2-outdegrees \cite{2013ARS, 2013LAA}.

In 2014, Hong-You \cite{2014} determine the digraphs with maximal signless Laplacian spectral radius among the strongly connected digraphs with given vertex connectivity. On the other hand, some extremal digraphs which attain the maximum or minimum spectral radius, the signless Laplacian spectral radius, the distance spectral radius,  or the distance signless Laplacian  spectral radius of digraphs with given parameters, such as given vertex connectivity, given arc connectivity,  given dichromatic number, given clique number, given girth  and so on, were characterized, see e.g. \cite{2013DM, 2013DAM, 2012DM, 2012DM2, 2013}.

In this paper, we give the spectrum of a matrix using the quotient matrix, and also apply these results to various matrices associated with graphs and digraphs as mentioned above. Some know results are improved.

\section{Some preliminaries}
\begin{defn}\label{defn21}{\rm(\cite{1979, 1986})}
Let $A=(a_{ij}), B=(b_{ij})$ be $n\times n$ matrices. If $a_{ij}\leq b_{ij}$ for all $i$ and $j$, then $A\leq B$.
If $A\leq B$ and $A\neq B$, then $A< B$. If $a_{ij} < b_{ij}$ for all $i$ and $j$, then $A \ll B$.
\end{defn}

\begin{lem}\label{lem22}{\rm(\cite{1979, 1986})}
Let $A, B$ be $n\times n$ matrices with the spectral radius $\rho(A)$ and $\rho(B)$. If $0\leq A\leq B$, then $\rho(A) \leq \rho(B)$.
Furthermore, if $B$ is irreducible and $0\leq A<B$, then $\rho(A) < \rho(B)$.
\end{lem}

By Lemma \ref{lem22}, we have the following  results  in terms of digraphs.

\begin{cor}\label{cor23}
Let $\overrightarrow{G}$ be a digraph and $\overrightarrow{H}$ be a spaning subdigraph of $\overrightarrow{G}$.
Then

{\rm (i) }$\rho(\overrightarrow{H})\leq \rho(\overrightarrow{G})$, $q(\overrightarrow{H})\leq q(\overrightarrow{G})$.

{\rm (ii) } If $\overrightarrow{G}$ is strongly connected, and $\overrightarrow{H}$ is a proper  subdigraph of $\overrightarrow{G}$,
then $\rho(\overrightarrow{H})< \rho(\overrightarrow{G})$, $q(\overrightarrow{H})< q(\overrightarrow{G})$.

{\rm (iii) } If $\overrightarrow{G}$ and $\overrightarrow{H}$ are  strongly connected,
then $\rho^{\mathcal{D}}(\overrightarrow{H})\geq \rho^{\mathcal{D}}(\overrightarrow{G})$,
          $q^{\mathcal{D}}(\overrightarrow{H})\geq q^{\mathcal{D}}(\overrightarrow{G})$.

{\rm (iv) } If $\overrightarrow{H}$ is a proper  subdigraph of $\overrightarrow{G}$,
then $\rho^{\mathcal{D}}(\overrightarrow{H})> \rho^{\mathcal{D}}(\overrightarrow{G})$,
     $q^{\mathcal{D}}(\overrightarrow{H})> q^{\mathcal{D}}(\overrightarrow{G})$.

The theorem for undirected graph is also established.
\end{cor}

\begin{lem}\label{lem24}{\rm(\cite{1979})}
If $A$ is an $n\times n$ nonnegative matrix with the spectral radius $\rho(A)$ and row sums $r_1, r_2, \ldots, r_n$, then
$\min\limits_{1\leq i\leq n}r_i\leq \rho(A)\leq \max\limits_{1\leq i\leq n}r_i$. Moreover, if $A$ is irreducible, then one of the equalities holds if and only if the row sums of $A$ are all equal.
\end{lem}

By Lemma \ref{lem24}, we have  $\rho(\overset{\longleftrightarrow}{K_n})=\rho^{\mathcal{D}}(\overset{\longleftrightarrow}{K_n})=n-1$, $q(\overset{\longleftrightarrow}{K_n})=q^{\mathcal{D}}(\overset{\longleftrightarrow}{K_n})=2(n-1)$;
$\rho(\overrightarrow{C_n})=1$, $q(\overrightarrow{C_n})=2$, $\rho^{\mathcal{D}}(\overrightarrow{C_n})=\frac{n(n-1)}{2}$, $q^{\mathcal{D}}(\overrightarrow{C_n})=n(n-1)$. Then by Corollary \ref{cor23}, we have

\begin{cor}\label{cor25}
Let $\overrightarrow{G}$ be a strongly connected digraph with $n$ vertices. Then
$$\rho(\overrightarrow{G})\leq n-1, \quad q(\overrightarrow{G})\leq 2n-2, \quad \rho^{\mathcal{D}}(\overrightarrow{G})\geq n-1,
\quad q^{\mathcal{D}}(\overrightarrow{G})\geq 2n-2,$$
\noindent  with equality holds if and only if $\overrightarrow{G}\cong \overset{\longleftrightarrow}{K_n}$.
\end{cor}

\begin{cor}\label{cor26}
Let $\overrightarrow{G}$ be a strongly connected digraph with $n$ vertices. Then
$$\rho(\overrightarrow{G})\geq 1, \quad q(\overrightarrow{G})\geq 2, \quad \rho^{\mathcal{D}}(\overrightarrow{G})\leq \frac{n(n-1)}{2},
\quad q^{\mathcal{D}}(\overrightarrow{G})\leq n(n-1),$$
\noindent  with equality holds if and only if $\overrightarrow{G}\cong \overrightarrow{C_n}.$
\end{cor}
\begin{proof}
 In \cite{2013DAM}, Theorem 3.2 show that $\rho^{\mathcal{D}}(\overrightarrow{G})\leq \frac{n(n-1)}{2},$
 and the equality holds if and only if $\overrightarrow{G}\cong \overrightarrow{C_n}.$
Now we only show $q^{\mathcal{D}}(\overrightarrow{G})\leq n(n-1)$ and the equality holds if and only if $\overrightarrow{G}\cong \overrightarrow{C_n}.$

 If $\overrightarrow{G}$ has a Hamiltonian dicycle,
 we have $ q^{\mathcal{D}}(\overrightarrow{G})\leq  q^{\mathcal{D}}(\overrightarrow{C_n})$
and the equality holds if and only if $\overrightarrow{G}\cong \overrightarrow{C_n}$ by Corollary \ref{cor23}.

If $\overrightarrow{G}$ does not contain a Hamiltonian dicycle. Noting that $\max\limits_{1\leq i\leq n}r_i\leq n(n-1)$.
If $\max\limits_{1\leq i\leq n}r_i< n(n-1)$, then $q^{\mathcal{D}}(\overrightarrow{G})\leq \max\limits_{1\leq i\leq n}r_i<n(n-1)= q^{\mathcal{D}}(\overrightarrow{C_n})$ by Lemma \ref{lem24}.

If $\max\limits_{1\leq i\leq n}r_i= n(n-1)$, then $\overrightarrow{G}$ contains a vertex $v_1$ such that $2Tr_{\overrightarrow{G}}(v_1)= n(n-1),$ then $\overrightarrow{G}$ contains a Hamiltonian dipath $P$ initiating at $v_1$. Suppose that $P= v_1\rightarrow v_2\rightarrow \ldots\rightarrow v_n$ is the Hamiltonian dipath initiating at $v_1$. Then there is no arc $(v_i, v_j)\in E(\overrightarrow{G})$ if $j-i\geq 2$ since $Tr_{\overrightarrow{G}}(v_1) = \frac{n(n-1)}{2}$. Since $\overrightarrow{G}$ is strongly connected and does not a Hamiltonian dicycle, there exists a dipath $P'$ from $v_n$ to $v_1$ and thus there exists some vertex, namely, $v_k (k\neq n)$,  is adjacent to $v_1$, that is $(v_k, v_1)\in E(\overrightarrow{G})$. Since $v_k$ is on the Hamiltonian dipath $P$, we have $(v_k, v_{k+1})\in E(\overrightarrow{G}).$ Hence
$$r_k \leq 2(1+1+2+\ldots +n-2)<2(1+2+\ldots +n-1)= 2Tr_{\overrightarrow{G}}(v_1)=n(n-1),$$  it implies that
the row sums of $\mathcal{Q}(\overrightarrow{G})$ are not equal.
Then by Lemma \ref{lem24}, we have
$$q^{\mathcal{D}}(\overrightarrow{G})< q^{\mathcal{D}}(\overrightarrow{C_n}).$$

Combining the above arguments, we complete the proof.
\end{proof}

\section{The spectrum of a matrix}
\hskip.6cm Let $I_p$ be the $p\times p$ identity matrix and $J_{p,q}$ be the $p\times q$ matrix in which every entry is $1$, or simply $J_p$ if $p=q$. Let $M$ be a matrix of order $n$, $\sigma(M)$ be the spectrum of the matrix $M$, $P_M(\lambda)=det(xI_n-M)$ be the characteristic polynomial of matrix $M$.

\begin{defn}\label{defn31}{\rm(\cite{2014a})}
Let $M$ be a real matrix of order $n$ described in the following block form

\begin{equation}\label{eq31}
M = \left(\begin{array}{ccc}
M_{11} & \cdots & M_{1t}\\
\vdots &  \ddots      &\vdots \\
 M_{t1}& \cdots & M_{tt}\\
\end{array}\right),
\end{equation}

\noindent where the diagonal blocks $M_{ii}$ are $n_i\times n_i$ matrices for any $i\in\{1,2,\ldots, t\}$ and $n=n_1+\ldots+n_t$.
For any $i,j\in\{1,2,\ldots, t\}$, let $b_{ij}$ denote the average row sum of $M_{ij}$, i.e. $b_{ij}$ is the sum of all entries in $M_{ij}$ divided by the number of rows. Then $B(M) = (b_{ij})$ (simply by $B$) is called the quotient matrix of $M$. If in addition for each pair $i, j$, $M_{ij}$ has constant row sum, then $B(M)$ is called the equitable quotient matrix of $M$.
\end{defn}

\begin{lem}\label{lem32}%{\rm(\cite{2011}, Chapter 2)}
Let $M=(m_{ij})_{n\times n}$ be defined as (\ref{eq31}),  and for any $i,j \in\{ 1,2\ldots,t\}$, the row sum of each block $M_{ij}$ be constant.
Let $B=B(M)=(b_{ij})$ be the equitable quotient matrix of $M$, and $\lambda$ be an  eigenvalue of $B$. Then $\lambda$  is also  an eigenvalue of $M$.
\end{lem}

\begin{proof}
Let $By = \lambda y$ where $y = (y_1,y_2,\ldots,y_t)^T$.
Define $Y = (y_{11},\ldots,y_{1,n_1},\ldots,y_{t1},\ldots,y_{t,n_t})^T$ by the relation $y_{i1} = y_{i2} = \ldots = y_{i,n_i}=y_i$  for each $i\in\{1,2,\ldots,t\}$. For any $i\in\{1,2,\ldots,t\}$ and $k\in\{1,2,\ldots, n_i\}$,
 let $M_i(k)$ be the $k$-th row of the $i$-th row blocks $(M_{i1}, \ldots, M_{it})$,
that is, $M_i(k)$ is the $l$-th row of $M$ where $l=n_1+\ldots+n_{i-1}+k$,
then by $M_i(k)Y=(MY)_l=\sum\limits_{j=1}^{n_1}m_{lj}y_1+\sum\limits_{j=n_1+1}^{n_1+n_2}m_{lj}y_2+\ldots
+\sum\limits_{j=n_1+\ldots+n_{t-1}+1}^{n_1+\ldots+n_t}m_{lj}y_t$ and the definition of $b_{ij}$ for each $i,j\in\{1,2,\ldots t\}$,
%$b_{i1}=\sum\limits_{j=1}^{n_1}m_{lj}},$
 we have
$$\lambda Y_l=\lambda y_{ik}=\lambda y_i=(By)_i= \sum \limits_{j=1 }^{t}b_{ij}y_j=M_i(k)Y=(MY)_{l},$$
%is equal to the $k$-th row of the $i$-th row blocks $(M_{i1}, \ldots, M_{it})$.
%$(Mx)_{ik} = b_{ii}y_i + \sum \limits_{j=1,j\not = i}^{t}b_{i,j}y_j = (By)_i = \lambda y_i = \lambda x_{ik}$,
%where $(Mx)_{ik}$ denoted the $k$-th row of $(M_{i1}, M_{i2},\ldots, M_{it})$, $1\leq k\leq n_i$,
thus we have $MY = \lambda Y,$  and we complete the proof.
\end{proof}

\begin{exam}\label{exam33}
Let $G=(V,E)$ be the Petersen graph as Figure 1. Let $\{V_1,V_2\}$ be a partition of $V=\{1,2,\ldots, 10\}$,
where $V_1=\{1,2,3,4,5\}$ and $V_2=\{6,7,8,9,10\}$.
 Then the equitable quotient matrices $B(A), B(L), B(Q), B(\mathcal{D}),
B(\mathcal{L}), B(\mathcal{Q})$ corresponding to the adjacency matrix $A(G)$,
the Laplacian matrix $L(G)$, the signless Laplacian matrix $Q(G)$, the distance matrix $\mathcal{D}(G)$,
the distance  Laplacian matrix $\mathcal{L}(G)$, the distance signless Laplacian matrix $\mathcal{Q}(G)$, respectively, are as follows:

\vspace{15mm}
\setlength{\unitlength} {4mm}
\begin{center}
\begin{picture}(15,10)

\put(4,4){\circle* {0.3}}   \put(10,4){\circle* {0.3}}    \put(2,8){\circle* {0.3}}       \put(12,8){\circle* {0.3}}
\put(7,12){\circle* {0.3}}   \put(5,6){\circle* {0.3}}     \put(9,6){\circle* {0.3}}       \put(4,8){\circle* {0.3}}
\put(10,8){\circle* {0.3}}   \put(7,10){\circle* {0.3}}

\put(4,4){\line(1,0){6}}    \put(4,4){\line(-1,2){2}}    \put(10,4){\line(1,2){2}}    \put(2,8){\line(5,4){5}}   \put(12,8){\line(-5,4){5}}
\put(2,8){\line(1,0){2}}    \put(10,8){\line(1,0){2}}    \put(7,10){\line(0,1){2}}     \put(4,4){\line(1,2){1}}   \put(10,4){\line(-1,2){1}}
\put(5,6){\line(1,2){2}}    \put(5,6){\line(5,2){5}}     \put(4,8){\line(1,0){6}}     \put(9,6){\line(-5,2){5}}   \put(9,6){\line(-1,2){2}}

\put(7,12.5){1}     \put(1,8){2}      \put(4,3){3}       \put(10,3){4}    \put(12.5,8){5}      %\put(15,9){$V_1=\{1,2,3,4,5\}$}
\put(7.5,10){6}     \put(3.5,8.2){7}    \put(4,6){8}    \put(9.5,6){9}    \put(9.5,8.2){10}    %\put(15,7){$V_2=\{6,7,8,9,10\}$}

\put(1,1.5){Figure $1$. The Petersen graph }
\end{picture}
\end{center}
\vskip-0.5cm

 $$B(A) = \left(\begin{array}{lcr}
2 & 1\\
1 & 2\\
\end{array}\right),  \qquad
 B(L) = \left(\begin{array}{lcr}
1 & -1\\
-1 & 1\\
\end{array}\right), \qquad
B(Q) = \left(\begin{array}{lcr}
5 & 1\\
1 & 5\\
\end{array}\right),$$

$$B(\mathcal{D}) = \left(\begin{array}{lcr}
6 & 9\\
9 & 6\\
\end{array}\right),  \qquad
 B(\mathcal{L})= \left(\begin{array}{lcr}
9 & -9\\
-9 & 9\\
\end{array}\right), \qquad
B(\mathcal{Q}) = \left(\begin{array}{lcr}
21 & 9\\
9 & 21\\
\end{array}\right).
$$

\noindent Then
$$\rho(B(A))=3, \rho(B(L))=2, \rho(B(Q))=6, \rho(B(\mathcal{D}))=15, \rho(B(\mathcal{L}))=18, \rho(B(\mathcal{Q}))=30, $$
but by directly calculating, we have  $$\rho(G) = 3,  \mu(G)=5, q(G) = 6, \rho^{\mathcal{D}}(G) = 15, \mu^{\mathcal{D}}(G))=18,
q^{\mathcal{D}}(G) = 30.$$
We see that the largest eigenvalue of the equitable quotient matrix $B(M)$ is the   largest eigenvalue of $M$ when $M$ is
the adjacency matrix $A(G)$, the signless Laplacian matrix $Q(G)$, the distance matrix $\mathcal{D}(G)$,
the distance  Laplacian matrix $\mathcal{L}(G)$ or
the distance signless Laplacian matrix $\mathcal{Q}(G)$  of a graph $G$,
and the result is totally different when $M$ is the Laplacian matrix $L(G)$   of a graph $G$.
\end{exam}

\begin{lem}\label{lem34}
Let $M$ be defined as (\ref{eq31}), and for any $i,j \in\{ 1,2\ldots,t\}$,
 $M_{ii} = l_iJ_{n_i} + p_iI_{n_i},$ $M_{ij} = s_{ij}J_{n_i,n_j}$ for $i\not= j$, where $l_i, p_i, s_{ij}$ are real numbers,
  $B=B(M)$ be the  quotient matrix of $M$.  Then
\begin{equation}\label{eq32}
\sigma(M)=\sigma(B)\cup \{p_i^{[n_i-1]} \mid i = 1,2\ldots,t\},
\end{equation}
 where $\lambda^{[t]}$ means that $\lambda$ is an eigenvalue with multiplicity $t$.
\end{lem}

\begin{proof}
It is obvious that for any $i,j \in\{ 1,2\ldots,t\}$, $M_{ij}$ has constant row sum, so $B$ is the equitable quotient matrix of $M$.
Then $\sigma(B)\subseteq \sigma(M)$ by Lemma \ref{lem32}.

On the other hand, we note that $\sigma(l_iJ_{n_i} + p_iI_{n_i}) = \{l_in_i + p_i, p_i^{[n_i-1]}\}$, where $l_iJ_{n_i} + p_iI_{n_i}$ has the all-one vector $J_{n_i,1}$ such that $(l_iJ_{n_i} + p_iI_{n_i})J_{n_i,1} = (l_in_i + p_i)J_{n_i,1}$,
and its all other eigenvectors corresponding to eigenvalue $p_i$ are orthogonal to $J_{n_i,1}$.

Let $x$ be an any eigenvector such that $(l_iJ_{n_i} + p_iI_{n_i})x = p_ix$, then $x^TJ_{n_i,1} = 0$, and
$(\mathbf{0}_{1,n_1},\ldots,x^T,\ldots,\mathbf{0}_{1,n_t})^T$ is an  eigenvector of $M$ corresponding to eigenvalue $p_i$.  Therefore the $p_i$ is  an eigenvalue of $M$ with multiplicities at least $n_i-1$.
And thus we obtain at least $\sum_{i=1}^{t}(n_i - 1) = n-t$ eigenvalues of $M$, that is,
$\{p_1^{[n_1-1]}, \ldots, p_t^{[n_t-1]}\}\subseteq \sigma (M)$.

Therefore $\sigma(B)\cup \{p_1^{[n_1-1]}, \ldots, p_t^{[n_t-1]}\}\subseteq \sigma (M)$ by  Lemma \ref{lem32},
and $|\sigma (M)|\leq |\sigma(B)|+|\{p_1^{[n_1-1]}, \ldots, p_t^{[n_t-1]}\}|=n$
by $|\sigma(B)|=t$ and $|\{p_1^{[n_1-1]}, \ldots, p_t^{[n_t-1]}\}|=n-t$.

If there exists some $p_i$ such that $p_i\in \sigma(B)$ where $i\in\{1,2,\ldots, t\}$, % the equitable quotient matrix $B$ also has eigenvalues $p_i$,
by the proof of Lemma \ref{lem32}, we have $My=p_iy$ with $y=(y_{11},\ldots,y_{1,n_1},\ldots,y_{t1},\ldots,y_{t,n_t})^T$,
where $y_{i1} = y_{i2} = \ldots = y_{i,n_i}=y_i$ for each $i\in\{1,2,\ldots,t\}$. Then we have $(\mathbf{0}_{1,n_1},\ldots,x^T,\ldots,\mathbf{0}_{1,n_t})y$$=y_i(x^TJ_{n_i,1})=0$,
it implies that  the eigenvectors corresponding to the eigenvalue $p_i$ of $B$ and the eigenvalue $p_i$ in
$\{p_1^{[n_1-1]}, \ldots, p_t^{[n_t-1]}\}$ are all orthogonal, then
 $|\sigma (M)|= |\sigma(B)|+|\{p_1^{[n_1-1]}, \ldots, p_t^{[n_t-1]}\}|=n$ and thus (\ref{eq32}) holds.
%%Therefore $\sigma(B)\cup \{p_1^{[n_1-1]}, \ldots, p_t^{[n_t-1]}\}\subseteq \sigma (M)$,
%and thus (\ref{eq32}) holds by $|\sigma(B)|=t$, $|\{p_1^{[n_1-1]}, \ldots, p_t^{[n_t-1]}\}|=n-t$ and $|\sigma (M)|=n$.
\end{proof}

%\begin{cor}\label{cor34}{\rm(\cite{1976})}
%Any divisor of a graph $G$ has the largest eigenvalue of $G$ as an eigenvalue.
%\end{cor}

\begin{exam}\label{exam35}
Let $G = K_{n_1,n_2,\ldots,n_t}$ be a complete t-partite graph with $n$ vertices for $ t \geq 2$,
% (if $t=n$, $G$ is the complete graph $K_n$ ),
the adjacency matrix $A=A(G)$, the Laplacian matrix $L=L(G)$, the signless Laplacian matrix $Q=Q(G)$,
the distance matrix $\mathcal{D}=\mathcal{D}(G),$  the distance Laplacian matrix $\mathcal{L}=\mathcal{L}(G)$
and the distance signless Laplacian matrix $\mathcal{Q}(G)$  of $G=K_{n_1,n_2,\ldots,n_t}$ are as follows:

(1). $A = M$, where $l_i = p_i = 0, s_{ij} = 1$ for $i\not= j$ where  $i,j\in\{ 1,2,\ldots,t\}.$

(2). $L =M$, where $l_i = 0, p_i = n-n_i, s_{ij} = -1$ for $i\not= j$ where  $i,j\in\{ 1,2,\ldots,t\}.$

(3). $Q = M$, where $l_i = 0, p_i = n-n_i, s_{ij} = 1$ for $i\not= j$ where  $i,j\in\{ 1,2,\ldots,t\}.$

(4). $\mathcal{D} = M$, where $l_i = 2, p_i = -2, s_{ij} = 1$ for $i\not= j$ where  $i,j\in\{ 1,2,\ldots,t\}.$

(5). $\mathcal{L} = M$, where $l_i = -2, p_i = n + n_i, s_{ij} = -1$ for $i\not= j$ where  $i,j\in\{ 1,2,\ldots,t\}.$

(6). $\mathcal{Q} = M$, where $l_i = 2, p_i = n + n_i - 4, s_{ij} = 1$ for $i\not= j$ where  $i,j\in\{ 1,2,\ldots,t\}.$

It is obvious that for any $i,j \in\{ 1,2\ldots,t\}$, $M_{ij}$ has constant row sum. % so $B$ is the equitable quotient matrix of $M$.\noindent
Then the corresponding equitable quotient matrices are as follows:

 $$B(A) = \left(\begin{array}{cccc}
0    & n_2   &\cdots   & n_t\\
n_1  & 0     & \cdots  & n_t\\
\vdots & \vdots &\ddots & \vdots \\
n_1 &    n_2 & \cdots &0 \\
  \end{array}\right), \qquad
  B(L) = \left(\begin{array}{cccc}
n-n_1 & -n_2 &\cdots  &-n_t\\
-n_1 & n-n_2 & \cdots  &-n_t\\
\vdots & \vdots &\ddots & \vdots \\
-n_1 & -n_2 & \cdots & n-n_t \\
  \end{array}\right),$$

  $$B(Q) = \left(\begin{array}{cccc}
n-n_1 & n_2 &\cdots  &n_t\\
n_1 & n-n_2 & \cdots  &n_t\\
\vdots & \vdots & \ddots &\vdots \\
n_1 & n_2 & \cdots &n-n_t \\
  \end{array}\right),
  \quad
B(\mathcal{D})= \left(\begin{array}{cccc}
2n_1-2 & n_2 &\cdots  &n_t\\
n_1 & 2n_2-2 & \cdots  &n_t\\
\vdots & \vdots & \ddots &\vdots \\
n_1 & n_2 & \cdots &2n_t-2\\
  \end{array}\right),$$

$$ B(\mathcal{L}) = \left(\begin{array}{cccc}
n-n_1 & -n_2 &\cdots  &-n_t\\
-n_1 & n-n_2 & \cdots & -n_t\\
\vdots & \vdots & \ddots &\vdots \\
-n_1 & -n_2 & \cdots &n-n_t \\
  \end{array}\right),$$

 $$ B(\mathcal{Q})= \left(\begin{array}{cccc}
n+3n_1-4 & n_2 &\cdots & n_t\\
n_1 & n+3n_2-4 & \cdots & n_t\\
\vdots & \vdots &\ddots & \vdots \\
n_1 & n_2 & \cdots &n+3n_t-4 \\
  \end{array}\right).$$

By Lemma \ref{lem34}, we have

(1). $ P_A(\lambda) =\lambda^{n-t} P_{B(A)}(\lambda) = \lambda^{n-t}[\prod\limits_{i=1}^t(\lambda + n_i) - \sum\limits_{i=1}^tn_i \prod \limits_{j=1,j \neq i}^t(\lambda + n_j)].$

(2). $ P_L(\lambda) =\prod\limits_{i=1}^t(\lambda - n + n_i)^{n_i-1} P_{B(L)}(\lambda) =\lambda(\lambda-n)^{t-1} \prod\limits_{i=1}^t(\lambda - n + n_i)^{n_i-1}.$

(3). $ P_Q(\lambda) = \prod\limits_{i=1}^t(\lambda - n + n_i)^{n_i-1}P_{B(Q)}(\lambda)$

        \hskip2.0cm           $= \prod\limits_{i=1}^t(\lambda - n + n_i)^{n_i-1}[\prod\limits_{i=1}^t(\lambda - n + 2n_i) - \sum\limits_{i=1}^tn_i \prod \limits_{j=1,j \neq i}^t(\lambda - n + 2n_j)].$

(4). $  P_{\mathcal{D}}(\lambda) = (\lambda + 2)^{n-t}P_{B(\mathcal{D})}(\lambda)$

\hskip2.0cm   $= (\lambda + 2)^{n-t}[\prod\limits_{i=1}^t(\lambda - n_i + 2) - \sum\limits_{i=1}^tn_i \prod \limits_{j=1,j \neq i}^t(\lambda - n_j + 2)]. \hskip.4cm {\rm(\cite{2013LAA2})}$

(5). $P_{\mathcal{L}}(\lambda) =\prod\limits_{i=1}^t(\lambda - n - n_i)^{n_i-1} P_{B(\mathcal{L})}(\lambda) =\lambda(\lambda-n)^{t-1} \prod\limits_{i=1}^t(\lambda - n - n_i)^{n_i-1}.$

 (6). $P_{\mathcal{Q}}(\lambda) = \prod\limits_{i=1}^t(\lambda - n - n_i + 4)^{n_i-1}P_{B(\mathcal{Q})}(\lambda)$

 \hskip2.0cm    $= \prod\limits_{i=1}^t(\lambda - n - n_i + 4)^{n_i-1}[\prod\limits_{i=1}^t(\lambda - n - 2n_i + 4) - \sum\limits_{i=1}^tn_i \prod \limits_{j=1,j \neq i}^t(\lambda - n - 2n_j + 4)].$

 It is obvious that  we obtain the spectrums of $L$ and $\mathcal{L}$ immediately.
In fact,  $\sigma(L) =\{0, n^{[t-1]}, (n-n_i)^{[n_i-1]}, i\in\{1, 2, \ldots, t\}\}$,
and $\sigma(\mathcal{L}) = \{0, n^{[t-1]}, (n+n_i)^{[n_i-1]}, i\in\{1, 2, \ldots, t\}\}.$
\end{exam}
\vskip.3cm

A block of $G$ is a maximal connected subgraph of $G$ that has no cut-vertex. A graph $G$ is a clique tree if each block of $G$ is a clique.
We call $\mathbb{K}_{u,n_2,\ldots,n_{k+1}}$ is a clique star if we replace each edge of the star $K_{1,k}$ by a clique $K_{n_i}$ such that $V(K_{n_i})\cap V(K_{n_j}) = u$ for $i \neq j$ and $i,j\in\{ 2,\ldots,k+1\}.$

\begin{exam}\label{exam36}
 Let $G=\mathbb{K}_{u,n_2,\ldots,n_{k+1}}$, where $n_1=|\{u\}| = 1$, $n_i \geq 2$ for any $i\in\{2,\ldots,k+1\}$ and
 $n = n_1 + n_2 + n_3 + \ldots + n_{k+1} - k$. Then
 the adjacency matrix $A=A(G)$, the Laplacian matrix $L=L(G)$, the signless Laplacian matrix $Q=Q(G)$,
the distance matrix $\mathcal{D}=\mathcal{D}(G),$  the distance Laplacian matrix $\mathcal{L}=\mathcal{L}(G)$
and the distance signless Laplacian matrix $\mathcal{Q}(G)$  of
 $G=\mathbb{K}_{u,n_2,\ldots,n_{k+1}}$ are as follows.

(1). $A=M$, where $l_1 = p_1 = 0$ and  $l_i = 1, p_i = -1$ for $i\not=1$,  $s_{ij} = 1$ for $i = 1 \mbox{or } j = 1$,
and $s_{ij} = 0$ for any $i,j\in\{2,\ldots,k+1\}$ and $i \neq j$.

(2). $L = M$, where $l_1 = n-1, p_1 = 0$ and $l_i = -1, p_i = n_i$ for $i\not=1$,
$s_{ij} = -1$ for $i = 1 \mbox{or } j = 1$,
and $s_{ij} = 0$ for any $i,j\in\{2,\ldots,k+1\}$ and $i \neq j$.

(3). $Q = M$, where $l_1 = n-1, p_1 = 0$ and $l_i = 1, p_i = n_i-2$ for $i\not=1$,
$s_{ij} = 1$ for $i = 1 \mbox{or } j = 1$,
and $s_{ij} = 0$ for any $i,j\in\{2,\ldots,k+1\}$ and $i \neq j$.

(4). $\mathcal{D} = M$, where $l_1 = 0, p_1 = 0$ and $l_i = 1, p_i = -1$ for $i\not=1$,
$s_{ij} = 1$ for $i = 1 \mbox{or } j = 1$,
and $s_{ij} = 2$ for any $i,j\in\{2,\ldots,k+1\}$ and $i \neq j$.

(5) $\mathcal{L} = M$, where $l_1 = n-1, p_1 = 0$ and $l_i = -1, p_i = 2n-n_i$ for $i\not=1$,
$s_{ij} = -1$ for $i = 1 \mbox{or } j = 1$,
and $s_{ij} = -2$ for any $i,j\in\{2,\ldots,k+1\}$ and $i \neq j$.

(6). $\mathcal{Q} = M$, where $l_1 = n-1, p_1 = 0$ and $l_i = 1, p_i = 2n-n_i-2$ for $i\not=1$,
$s_{ij} = 1$ for $i = 1 \mbox{or } j = 1$,
and $s_{ij} = 2$ for any $i,j\in\{2,\ldots,k+1\}$ and $i \neq j$.

 It is obvious that for any $i,j\in\{2,\ldots,k+1\}$, $M_{ij}$ has constant row sum.
Then the corresponding equitable quotient matrices are as follows:

$$B(A) = \left(\begin{array}{cccc}
0 & n_1-1 &\cdots  &n_k-1\\
1 & n_1-2 & \cdots & 0\\
\vdots & \vdots & \ddots &\vdots \\
1 & 0 & \cdots &n_k-2 \\
  \end{array}\right), \qquad
B(L) = \left(\begin{array}{cccc}
n-1 & 1-n_1 &\cdots & 1-n_k\\
-1 & 1 & \cdots & 0\\
\vdots & \vdots & \vdots\ddots &\vdots \\
-1 & 0 & \cdots &1 \\
  \end{array}\right),$$

 $$B(Q) = \left(\begin{array}{cccc}
n-1 & n_1-1 &\cdots & n_k-1\\
1 & 2n_1-3 & \cdots & 0\\
\vdots & \vdots & \ddots & \vdots \\
1 & 0 & \cdots &2n_k-3 \\
  \end{array}\right), \qquad
B(\mathcal{D}) = \left(\begin{array}{cccc}
0 & n_1-1 &\cdots  &n_k-1\\
1 & n_1-2 & \cdots & 2£¨n_k-1£©\\
\vdots & \vdots & \ddots & \vdots \\
1 & 2£¨n_1-1£© & \cdots &n_k-2\\
  \end{array}\right), $$

$$B(\mathcal{L}) = \left(\begin{array}{cccc}
n-1 & 1-n_1 &\cdots  &1-n_k\\
-1 & 2n-2n_1+1 & \cdots  &-2(n_k-1)\\
\vdots & \vdots & \ddots &\vdots \\
-1 & -2(n_1-1) & \cdots &2n-2n_k+1 \\
  \end{array}\right),$$

$$B(\mathcal{Q}) = \left(\begin{array}{cccc}
n-1 & n_1-1 &\cdots & n_k-1\\
1 & 2n-3 & \cdots &2£¨n_k-1£©\\
\vdots & \vdots & \ddots &\vdots \\
1 & 2(n_1-1) & \cdots & 2n-3 \\
  \end{array}\right).$$

By Lemma \ref{lem34}, we have

(1). $P_A(\lambda) = (\lambda+1)^{n-k-1}P_{B(A)}(\lambda)$

\hskip2cm $= (\lambda+1)^{n-k-1}[\lambda\prod\limits_{i=2}^{k+1}(\lambda - n_i + 2) - \sum\limits_{i=2}^{k+1}(n_i-1) \prod \limits_{j=2,j \neq i}^{k+1}(\lambda - n_j +2)].$

(2). $P_L(\lambda) =(\lambda-n_i)^{n_i-2}P_{B(L)}(\lambda)= \lambda(\lambda-n)(\lambda-1)^{k-1}(\lambda-n_i)^{n_i-2}.$

(3). $P_Q(\lambda)=\prod\limits_{i=2}^{k+1}(\lambda - n_i + 2)^{n_i-2}P_{B(Q)}(\lambda)$

\hskip2cm $= \prod\limits_{i=2}^{k+1}(\lambda - n_i + 2)^{n_i-2}[\lambda\prod\limits_{i=2}^{k+1}(\lambda - 2n_i + 3) - \sum\limits_{i=2}^{k+1}(n_i-1) \prod \limits_{j=2,j \neq i}^{k+1}(\lambda - 2n_j +3)].$

(4). $ P_{\mathcal{D}}(\lambda) =(\lambda + 1)^{n-k-1}P_{B(\mathcal{D})}(\lambda)$

\hskip2cm $= (\lambda + 1)^{n-k-1}[\lambda\prod\limits_{i=2}^{k+1}(\lambda + n_i) - (2\lambda+1)\sum\limits_{i=2}^{k+1}(n_i-1)\prod \limits_{j=2,j \neq i}^{k+1}(\lambda + n_j)].$

(5). $P_{\mathcal{L}}(\lambda) =(\lambda -2n+n_i)^{n_i-2}P_{B(\mathcal{L})}(\lambda)= \lambda(\lambda-n)(\lambda-2n+1)^{k-1}(\lambda-2n+n_i)^{n_i-2}.$

(6). $P_{\mathcal{Q}}(\lambda) =\prod\limits_{i=2}^{k+1}(\lambda - 2n + n_i +2)^{n_i-2}P_{B(\mathcal{Q})}(\lambda)$

\hskip2cm  $= \prod\limits_{i=2}^{k+1}(\lambda - 2n + n_i +2)^{n_i-2}[(\lambda-n+1)\prod\limits_{i=2}^{k+1}(\lambda - 2n + 2n_i + 1)$

\hskip2cm $-(2\lambda-2n+3)\sum\limits_{i=2}^{k+1}(n_i-1) \prod \limits_{j=2,j \neq i}^{k+1}(\lambda - 2n + 2n_j + 1)].$

It is obvious that we can obtain the spectrum of $L$ and $\mathcal{L}$ immediately.  In fact,
 $\sigma(L) =\{0, n, 1^{[k-1]}, n_i^{[n_i-2]}, i\in\{2, 3, \ldots, k+1\}\}$ and
  $\sigma(\mathcal{L}) = \{0, n, (2n-1)^{[k-1]},(2n-n_i)^{[n_i-2]}, i\in \{2, 3, \ldots, k+1\}\}.$
\end{exam}

By Lemma \ref{lem32}, Example \ref{exam33} and Examples \ref{exam35}--\ref{exam36},  we  proposed the following conjecture for further research.

\begin{con}\label{con31}
Let $M$ be a nonnegative matrix, $B(M)$ be the equitable quotient matrix of $M$. Then
the largest eigenvalue of  $B(M)$ is the largest eigenvalue of  $M$.
\end{con}

Consider two sequences of real numbers: $\lambda_{1}\geq \lambda_{2} \geq ... \geq\lambda_{n}$, and $\mu_{1}\geq \mu_{2}\geq ...\geq \mu_{m}$  with $m<n$. The second sequence is said to interlace the first one whenever
$\lambda_{i}\geq \mu_{i}\geq \lambda_{n-m+i}$ for $i=1,2,...,m$. The interlacing is called tight if
there exists an integer $k \in [1,m]$ such that $\lambda_{i}=\mu_{i}$ hold for $1\leq i \leq k$ and $\lambda_{n-m+i}= \mu_{i}$ hold for $k + 1 \leq i \leq m$.

\begin{lem}{\rm(\cite{1995H})}\label{lem33}
Let $M$ be a symmetric matrix and  have the block form as (\ref{eq31}),
$B$ be the quotient matrix of $M$. Then

{\rm (1) } The eigenvalues of $B$ interlace the eigenvalues of $M$.

{\rm (2) } If the interlacing is tight, then $B$ is the equitable matrix of $M$.
\end{lem}

By Lemmas \ref{lem32}-\ref{lem33}, we have the following result immediately.

\begin{them}\label{thm34}%{\rm(\cite{2011}, Chapter 2)}
Let $M=(m_{ij})_{n\times n}$ be a symmetric matrix and defined as (\ref{eq31}),
$B=B(M)$ be the  quotient matrix of $M$, and  $\mu_{1}\geq \mu_{2}\geq ...\geq \mu_{m}$ be all  eigenvalues of $B$.
Then $\mu_{1}, \mu_{2}, ..., \mu_{m}$  are  eigenvalues of $M$ if and only if $B$ is the equitable matrix of $M$.
\end{them}

\vskip.3cm

\section{Spectral radius of strongly connected digraphs with given connectivity}

\hskip.6cm Let $\Omega(n,k)$ be the set of all simple strong connected digraphs on $n$ vertices with vertex connectivity $k$. Let $\overrightarrow{G}_1 \bigtriangledown \overrightarrow{G}_2$ denote the digraph $G=(V,E)$ obtained from two disjoint digraphs $\overrightarrow{G}_1$, $\overrightarrow{G}_2$ with vertex set $V=V(\overrightarrow{G}_1)\cup V(\overrightarrow{G}_2)$ and arc set $E = E(\overrightarrow{G}_1) \cup E(\overrightarrow{G}_2) \cup \{(u,v),(v,u)|u \in V(\overrightarrow{G}_1), v \in V(\overrightarrow{G}_2)\}.$

Let $p,k$ be integers with $1\leq k\leq n-2, 1\leq p\leq n-k-1,$  $\overrightarrow{K}(n,k,p)$ denote the digraph $\overrightarrow{K_k} \bigtriangledown (\overrightarrow{K_p} \cup \overrightarrow{K}_{n - p - k}) \cup E,$ where $E = \{(u,v)|u \in \overrightarrow{K_p}, v \in \overrightarrow{K}_{n - p - k}\}$. Clearly, $\overrightarrow{K}(n,k,p)\in \Omega(n,k).$
Then the adjacency matrix, the signless Laplacian matrix, the distance matrix, the distance signless Laplacian matrix of
 $\overrightarrow{K}(n,k,p)$ are as follows,  where $q=n-p-k$.

 $$A(\overrightarrow{K}(n,k,p)) = \left(\begin{array}{lcr}
  J_p-I_p &  J_{p,k} &  J_{p,q}\\
  J_{k,p} & J_k-I_k  &  J_{k,q}\\
  \mathbf{0}_{q,p} &  J_{q,k} & J_q-I_q\\
  \end{array}\right),$$
  $$Q(\overrightarrow{K}(n,k,p)) = \left(\begin{array}{lcr}
 J_p+(n-2)I_p &  J_{p,k} &  J_{p,q}\\
  J_{k,p} & J_k+(n-2)I_k &  J_{k,q}\\
  \mathbf{0}_{q,p} &  J_{q,k} &  J_q+(n-p-2)I_q\\
  \end{array}\right), $$
  $$\mathcal{D}(\overrightarrow{K}(n,k,p)) = \left(\begin{array}{lcr}
 J_p-I_p &  J_{p,k} &  J_{p,q}\\
  J_{k,p} & J_k-I_k &  J_{k,q}\\
  2J_{q,p} &  J_{q,k} &  J_q-I_q\\
  \end{array}\right),$$
   $$\mathcal{Q}(\overrightarrow{K}(n,k,p)) =\left(\begin{array}{lcr}
 J_p+(n-2)I_p &  J_{p,k} &  J_{p,q}\\
  J_{k,p} & J_k+(n-2)I_k &  J_{k,q}\\
  2J_{q,p} &  J_{q,k} &  J_q+(n+p-2)I_q\\
  \end{array}\right).$$
    \noindent

\vskip.3cm

\begin{prop}\label{prop41}{\rm(\cite{1976})}
Let $\overrightarrow{G}$ be a strongly connected digraphs with vertex connectivity $k$. Suppose that $S$ is a $k$-vertex cut of $\overrightarrow{G}$ and $\overrightarrow{G}_1, \overrightarrow{G}_2,\ldots, \overrightarrow{G}_t$ are the strongly connected components of $\overrightarrow{G}-S$. Then there exists an ordering of $\overrightarrow{G}_1, \overrightarrow{G}_2,\ldots, \overrightarrow{G}_t$ such that for $1\leq i\leq t$ and $v\in V(\overrightarrow{G}_i)$, every tail of $v$ is in $\bigcup \limits_{j=1}^{i-1}\overrightarrow{G}_j$.
\end{prop}

\begin{rem}\label{rem42}
By Proposition \ref{prop41}, we know that $\overrightarrow{G}_1$ is the strongly connected component of $\overrightarrow{G}-S$ where the inneighbors of vertices of $V(\overrightarrow{G}_1)$ in $\overrightarrow{G}-S-\overrightarrow{G}_1$ are zero. Let $\overrightarrow{G}_2=\overrightarrow{G}-S-\overrightarrow{G}_1$. We add arcs to $\overrightarrow{G}$ until both induced subdigraph of $V(\overrightarrow{G}_1)\cup S$ and induced subdigraph of $V(\overrightarrow{G}_2)\cup S$ attain to complete digraphs, add arc $(u,v)$ for any $u\in V(\overrightarrow{G}_1)$ and any $v\in V(\overrightarrow{G_2})$, the new digraph denoted by $\overrightarrow{H}$. Clearly, $\overrightarrow{H}=\overrightarrow{K}(n,k,p)\in \Omega(n,k)$ for some $p$ such that $1\leq p\leq n-k-1$.
Since $\overrightarrow{G}$ is the spanning subdigraph of $\overrightarrow{H}$,
then by Corollary \ref{cor23}, we have $\rho(\overrightarrow{G})\leq \rho(\overrightarrow{K}(n,k,p))$,
$q(\overrightarrow{G})\leq q(\overrightarrow{K}(n,k,p))$, $\rho^{\mathcal{D}}(\overrightarrow{G})\geq \rho^{\mathcal{D}}(\overrightarrow{K}(n,k,p))$
and  $q^{\mathcal{D}}(\overrightarrow{G})\geq q^{\mathcal{D}}(\overrightarrow{K}(n,k,p))$.
Thus the extremal digraphs which achieve the maximal  (signless Laplacian) spectral radius and the minimal
 distance (signless Laplacian) spectral radius  in $\Omega(n,k)$ must be some $\overrightarrow{K}(n,k,p)$ for $1\leq p\leq n-k-1$.
\end{rem}

\begin{them}
Let $n,k$ be given positive integers with $1\leq k\leq n-2$, $\overrightarrow{G}\in \Omega(n,k).$ Then

(i). {\rm(\cite{2012DM})} \begin{equation}\label{eq41}
\rho(\overrightarrow{G})\leq \frac{n-2+\sqrt{(n-2)^2+4k}}{2},
\end{equation}
with equality if and only if $\overrightarrow{G}\cong \overrightarrow{K}(n,k,1)$ or $\overrightarrow{G}\cong \overrightarrow{K}(n,k,n-k-1).$

(ii). {\rm(\cite{2014})} \begin{equation}\label{eq42}
q(\overrightarrow{G})\leq \frac{2n+k-3+\sqrt{(2n-k-3)^2+4k}}{2},
\end{equation}
with equality if and only if $\overrightarrow{G}\cong \overrightarrow{K}(n,k,n-k-1).$

(iii). {\rm(\cite{2012DM2})} \begin{equation}\label{eq43}
\rho^\mathcal{D}(\overrightarrow{G})\geq \frac{n-2+\sqrt{(n+2)^2-4k-8}}{2},
\end{equation}
with equality if and only if $\overrightarrow{G}\cong \overrightarrow{K}(n,k,1)$ or $\overrightarrow{G}\cong \overrightarrow{K}(n,k,n-k-1)$.

(iv). \begin{equation}\label{eq44}
q^\mathcal{D}(\overrightarrow{G})\geq \frac{3n-3+\sqrt{(n+3)^2-8k-16}}{2},
\end{equation}
with equality if and only if $\overrightarrow{G}\cong \overrightarrow{K}(n,k,1).$
\end{them}

\begin{proof}
Now we show (i) holds. We apply Lemma \ref{lem34} to $A=A(\overrightarrow{K}(n,k,p))$.
Since $t=3$, $l_i=1,p_i=-1$ for $1\leq i\leq 3$, $s_{31}=0$ and  $s_{ij}=1$ for others $i, j\in\{1,2,3\}$ and $i\not=j$, we have $\sigma(A)=\sigma(B(A))\cup\{(-1)^{[n-3]}\}$, where the corresponding equitable quotient matrix of $A$ is

$$B(A)= \left(\begin{array}{ccc}
 p-1 &  k  & q\\
 p  &  k-1 & q\\
 0  &  k & q-1\\
  \end{array}\right), $$
  the eigenvalues of $B(A)$ are $-1,\frac{n-2\pm\sqrt{4p^2-4(n-k)p+n^2}}{2}$.
Thus $\rho(A)=\frac{n-2+\sqrt{4p^2-4(n-k)p+n^2}}{2}$.

It is obvious that $\frac{n-2+\sqrt{4p^2-4(n-k)p+n^2}}{2}\leq \frac{n-2+\sqrt{(n-2)^2+4k}}{2}$, and equality holds if and only if $p=1$ or $p=n-k-1$. Thus (\ref{eq41}) holds and equality holds if and only if $\overrightarrow{G}\cong \overrightarrow{K}(n,k,1)$ or $\overrightarrow{G}\cong \overrightarrow{K}(n,k,n-k-1).$

Now we show (ii) holds.
Similarly, we apply Lemma \ref{lem34} to $Q=Q(\overrightarrow{K}(n,k,p))$.
Since $t=3$, $l_i=1$ for $1\leq i\leq 3$, $p_1=p_2=n-2$, $p_3=n-p-2,$ $s_{31}=0$   and  $s_{ij}=1$ for others $i, j\in\{1,2,3\}$ and $i\not=j$,
we have $\sigma(Q)=\sigma(B(Q))\cup\{(n-2)^{[p+k-2]},(n-p-2)^{[q-1]}\},$ where the corresponding equitable quotient matrix of $Q$ is

$$B(Q)= \left(\begin{array}{ccc}
 p+n-2 &  k  & q\\
 p  &  k+n-2 & q\\
 0  &  k & q+n-p-2\\
  \end{array}\right),$$ the eigenvalues of $B(Q)$ are $n-2,\frac{(3n-p-4)\pm\sqrt{(n-3p)^2+8pk}}{2}$.
  Thus $\rho(Q)=\frac{3n-p-4+\sqrt{(n-3p)^2+8pk}}{2}$.

By the same proof of Theorem 7.6 in \cite{2014}, we can show (\ref{eq42}) holds by proving
  $$\frac{3n-p-4+\sqrt{(n-3p)^2+8pk}}{2}\leq \frac{2n+k-3+\sqrt{(2n-k-3)^2+4k}}{2}$$   for $1\leq p\leq n-k-1$,
  and equality holds if and only if  $\overrightarrow{G}\cong \overrightarrow{K}(n,k,n-k-1).$

Now we show (iii) holds. We apply Lemma \ref{lem34} to $\mathcal{D}=\mathcal{D}(\overrightarrow{K}(n,k,p))$.
Since $l_i=1, p_i=-1$ for $1\leq i\leq 3$, $s_{31}=2$ and   $s_{ij}=1$ for others $i, j\in\{1,2,3\}$ and $i\not=j$, we have $\sigma(\mathcal{D})=\sigma(B(\mathcal{D}))\cup\{(-1)^{[n-3]}\},$ and the corresponding equitable quotient matrix of $\mathcal{D} $ is
$$B(\mathcal{D})= \left(\begin{array}{ccc}
 p-1 &  k  & q\\
 p  &  k-1 & q\\
 2p  &  k & q-1\\
  \end{array}\right),$$ the eigenvalues  of $B(\mathcal{D})$ are $-1,\frac{(n-2)\pm\sqrt{-4p^2+4(n-k)p+n^2}}{2}$.
Thus $\rho(\mathcal{D})=\frac{n-2+\sqrt{-4p^2+4(n-k)p+n^2}}{2}$.

It is obvious that $\frac{n-2+\sqrt{-4p^2+4(n-k)p+n^2}}{2}\geq \frac{n-2+\sqrt{(n+2)^2-4k-8}}{2}$,  and equality holds if and only if
$p=1$ or $p=n-k-1$. Thus (\ref{eq43}) holds and equality holds if and only if $\overrightarrow{G}\cong \overrightarrow{K}(n,k,1)$ or $\overrightarrow{G}\cong \overrightarrow{K}(n,k,n-k-1).$

Now we show (iv) holds.  We apply Lemma \ref{lem34} to $\mathcal{Q}=\mathcal{Q}(\overrightarrow{K}(n,k,p))$. Since $l_1=l_2=l_3=1,$ $p_1=p_2=n-2, $ $p_3=n+p-2,$ $s_{31}=2$ and   $s_{ij}=1$ for others $i, j\in\{1,2,3\}$ and $i\not=j$,  we have $\sigma(\mathcal{Q})=\sigma(B(\mathcal{Q}))\cup\{(n-2)^{[p+k-2]},(n+p-2)^{[q-1]}\},$
and the corresponding equitable quotient matrix of $\mathcal{Q} $ is
$$B(\mathcal{Q})= \left(\begin{array}{ccc}
 p+n-2 &  k  & q\\
 p  &  k+n-2 & q\\
 2p  &  k & q+n+p-2\\
  \end{array}\right),$$  the eigenvalues  of $B(\mathcal{Q})$ are $n-2,\frac{(3n+p-4)\pm\sqrt{(n+3p)^2-16p^2-8kp}}{2}$.
  Thus $$\rho(\mathcal{Q})= \frac{3n+p-4 + \sqrt{(n+3p)^2-16p^2-8kp}}{2}.$$

  Now we show $\frac{3n+p-4 + \sqrt{(n+3p)^2-16p^2-8kp}}{2}\geq \frac{3n-3+\sqrt{(n+3)^2-8k-16}}{2}$ for $1\leq p\leq n-k-1$, and the equality holds if and only if $p=1$.
  Let $f(p)=\frac{3n+p-4 + \sqrt{(n+3p)^2-16p^2-8kp}}{2}$. Then
$$\frac{\partial ^{2}f(p)}{\partial p^2} = \frac{-4((2n-k)(n-k)+k^2)}{((n+3p)^2-16p^2-8kp)^{\frac{3}{2}}}<0.$$

Thus, for fixed $n$ and $k$, the minimal value of $f(p)$ must be taken at either $p=1$ or $p=n-k-1$. Let $\alpha=k^2-6k-7+8n$ and $\beta=n^2+6n-7-8k$. Then
$$2[f(n-k-1) - f(1)]=n-k-2+\sqrt{\alpha}-\sqrt{\beta}=(n-k-2)(1-\frac{n+k}{\sqrt{\alpha}+\sqrt{\beta}}).$$

We can assume that $n > k+2$ since in case $k=n-2$ there is only one value of $p$ under consideration.
Now suppose that $f(n-k-1) - f(1)\leq 0$. We will produce a contradiction. We have
$$\sqrt{\alpha}+\sqrt{\beta}\leq n+k , \sqrt{\alpha}-\sqrt{\beta}\leq -n+k+2.$$
Whence $\sqrt{\alpha}\leq k+1$ and $\alpha\leq (k+1)^2$ which reduces to $k\geq n-1$ which is out of range. Thus
$f(n-k-1) > f(1)$ and $q^\mathcal{D}(\overrightarrow{G})\geq f(1)\texttt{}=q^\mathcal{D}(\overrightarrow{K}(n,k,1))=\frac{3n-3+\sqrt{(n+3)^2-8k-16}}{2}$,
with equality if and only if $\overrightarrow{G}\cong \overrightarrow{K}(n,k,1).$
\end{proof}

It is natural that  whether there exists similar result for
 the  Laplacian spectral radius or the  distance Laplacian spectral radius  in $\Omega(n,k)$ or not?
 In fact, we can obtain the spectrum of  the  Laplacian matrix or the  distance Laplacian matrix of $\overrightarrow{K}(n,k,p)$ immediately.

\begin{prop}\label{prop44}
Let $\overrightarrow{K}(n,k,p)$ defined as before. Then

(i). $\sigma(L(\overrightarrow{K}(n,k,p)))=\{0, n^{[p+k-1]}, (n-p)^{[q]}\}.$

(ii). $\sigma(\mathcal{L}(\overrightarrow{K}(n,k,p)))=\{0, n^{[p+k-1]}, (n+p)^{[q]}\}.$
\end{prop}
\begin{proof}
Firstly, the  Laplacian matrix $L(\overrightarrow{K}(n,k,p))$ and
the  distance Laplacian matrix $\mathcal{L}(\overrightarrow{K}(n,k,p))$ of $\overrightarrow{K}(n,k,p)$ are the following matrices,
where $q=n-p-k$.

 $$L=L(\overrightarrow{K}(n,k,p)) = \left(\begin{array}{ccc}
  -J_p+nI_p &  -J_{p,k} &  -J_{p,q}\\
  -J_{k,p} & -J_k+nI_k  &  -J_{k,q}\\
  \mathbf{0}_{q,p} &  -J_{q,k} & -J_q+(n-p)I_q\\
  \end{array}\right),$$

$$\mathcal{L}=\mathcal{L}(\overrightarrow{K}(n,k,p)) = \left(\begin{array}{ccc}
 -J_p+nI_p &  -J_{p,k} &  -J_{p,q}\\
  -J_{k,p} & -J_k+nI_k &  -J_{k,q}\\
  -2J_{q,p} &  -J_{q,k} &  -J_q+(n+p)I_q\\
  \end{array}\right).$$

Then the corresponding equitable quotient matrices are as follows:
$$B(L)= \left(\begin{array}{ccc}
 n-p &    -k  & -q\\
 -p  &    n-k  &  -q\\
 0   &   -k  &  k\\
  \end{array}\right), \qquad
B(\mathcal{L})= \left(\begin{array}{lcr}
 n-p &    -k  & -q\\
 -p  &    n-k  &  -q\\
 -2p   &   -k  &  n+p-q\\
  \end{array}\right).$$

  Then by Lemma \ref{lem34} and directly calculating, we obtain (i) and (ii).
\end{proof}

\vskip.3cm

\section{Spectral radius of connected graphs with given connectivity}

\hskip.6cm Let  $\mathcal{C}(n,k)$ be the set of all simple connected  graphs on $n$ vertices with vertex connectivity $k$.
Let ${G_1}\bigtriangledown  {G_2}$ denote the graph $G=(V,E)$ obtained from two disjoint graphs ${G_1}$, ${G_2}$ by joining each vertex of $G_1$ to each vertex of $G_2$ with vertex set $V=V({G}_1)\cup V({G}_2)$ and edge set $E = E({G}_1) \cup E({G}_2) \cup \{uv | u \in V( {G}_1), v \in V( {G}_2)\}$.

Let $p,k$ be integers with $1\leq k\leq n-2, 1\leq p\leq n-k-1,$ and  ${K}(n,k,p)$ be the graph
 ${K_k} \bigtriangledown ({K_p} \cup {K}_{n-p-k})$. Clearly, ${K}(n,k,p)\in \mathcal{C}(n,k).$
Then the adjacency matrix, the signless Laplacian matrix, the distance matrix, the distance signless Laplacian matrix of ${K}(n,k,p)$ are as follows, where $q=n-p-k$.

 $$A({K}(n,k,p)) = \left(\begin{array}{lcr}
  J_p-I_p &  J_{p,k} &  \mathbf{0}_{p,q}\\
  J_{k,p} & J_k-I_k  &  J_{k,q}\\
  \mathbf{0}_{q,p} &  J_{q,k} & J_q-I_q\\
  \end{array}\right),$$
  $$Q({K}(n,k,p)) = \left(\begin{array}{lcr}
 J_p+(p+k-2)I_p &  J_{p,k} &  \mathbf{0}_{p,q}\\
  J_{k,p} & J_k+(n-2)I_k &  J_{k,q}\\
  \mathbf{0}_{q,p} &  J_{q,k} &  J_q+(n-p-2)I_q\\
  \end{array}\right), $$
  $$\mathcal{D}({K}(n,k,p)) = \left(\begin{array}{lcr}
 J_p-I_p &  J_{p,k} &  2J_{p,q}\\
  J_{k,p} & J_k-I_k &  J_{k,q}\\
  2J_{q,p} &  J_{q,k} &  J_q-I_q\\
  \end{array}\right),$$
   $$\mathcal{Q}({K}(n,k,p)) =\left(\begin{array}{lcr}
 J_p+(n+q-2)I_p &  J_{p,k} &  2J_{p,q}\\
  J_{k,p} & J_k+(n-2)I_k &  J_{k,q}\\
  2J_{q,p} &  J_{q,k} &  J_q+(n+p-2)I_q\\
  \end{array}\right).$$
    \noindent

\vskip1cm

\begin{rem}\label{rem51}
Let ${G}$ be a  connected graphs with vertex connectivity $k$. Suppose that $S$ is a $k$-vertex cut of ${G}$,
%Let $S$ be a vertex subset of $G$ such that $|S| = k$ and $G - S$ is disconnected, where
and $G_1$ is a connected component of $G - S$. Let ${G_2}={G}-S-{G_1}$, we add edges to ${G}$ until both induced subgraph of $V({G_1})\cup S$ and induced subgraph of $V({G_2})\cup S$ attain to complete graphs, the new graph denoted by ${H}$. Clearly, ${H}={K}(n,k,p)\in \mathcal{C}(n,k)$ for some $p$ such that $1\leq p\leq n-k-1$. Since ${G}$ is the spanning subgraph of ${H}$, then by Corollary \ref{cor23},
we have $\rho({G})\leq \rho({K}(n,k,p))$,
$q({G})\leq q({K}(n,k,p))$, $\rho^{\mathcal{D}}({G})\geq \rho^{\mathcal{D}}({K}(n,k,p))$
and  $q^{\mathcal{D}}({G})\geq q^{\mathcal{D}}({K}(n,k,p))$.
Thus the extremal graphs which achieve the maximal  (signless Laplacian) spectral radius and the minimal
 distance (signless Laplacian) spectral radius  in $\mathcal{C}(n,k)$ must be some $K(n,k,p)$ for $1\leq p\leq n-k-1$.
\end{rem}

\begin{them}
Let $n,k$ be given positive integers with $1\leq k\leq n-2$, ${G}\in \mathcal{C}(n,k).$ Then

(i) {\rm(\cite{2010LAA})} $\rho({G})\leq \rho(K(n,k,1)),$ and $\rho({{K}(n,k,1)}$ is the largest root of equation (\ref{eq51}):
\begin{equation}\label{eq51}
\lambda^3 - (n - 3)\lambda^2 - (n + k - 2)\lambda + k(n - k - 2) = 0,
\end{equation}
with equality holds if and only if $G = K(n,k,1).$

(ii) {\rm(\cite{2010LAA})}
\begin{equation}\label{eq52}
q({G})\leq q(K(n,k,1)) = \frac{2n+k-4+\sqrt{(2n-k-4)^2+8k}}{2},
\end{equation}
with equality holds if and only if $G = K(n,k,1).$

(iii) {\rm(\cite{2012DM2})}
$\rho^\mathcal{D}({G})\geq \rho^\mathcal{D}(K(n,k,1)),$ and $\rho^\mathcal{D}({{K}(n,k,1)}$ is the largest root of equation (\ref{eq53}):
\begin{equation}\label{eq53}
\lambda^3 - (n - 3)\lambda^2 - (5n - 3k - 6)\lambda + kn - k^2 + 2k - 4n + 4= 0,
\end{equation}
with equality holds if and only if $G = K(n,k,1).$

(iv) {\rm(\cite{2013})}
 $q^\mathcal{Q}({G})\geq q^\mathcal{Q}(K(n,k,1)),$ and $q^\mathcal{Q}({{K}(n,k,1)}$ is the largest root of equation (\ref{eq54}):
\begin{equation}\label{eq54}
\lambda^3 - (5n - k -6)\lambda^2 + (8n^2 - 19kn - 24n + 8k + 16)\lambda - 4n^3 + 2(k + 10)n^2 - 2(5k + 16)n + 12k + 16 = 0,
\end{equation}
with equality holds if and only if $G = K(n,k,1).$
\end{them}

\begin{proof}
Firstly, we show (i) holds. We apply Lemma \ref{lem34} to $A=A({K}(n,k,p))$.
Since $t=3$, $l_1=l_2=l_3=1,$ $p_1=p_2=p_3=-1,$  $s_{13}=s_{31}=0$ and $s_{12}=s_{21}=s_{23}=s_{32}=1$, we have $\sigma(A)=\sigma(B(A))\cup\{(-1)^{[n-3]}\},$  where the corresponding equitable quotient matrix of  $A$ is
$B(A)=\left(\begin{array}{ccc}
 p-1 &  k  & 0\\
 p  &  k-1 & q\\
 0  &  k & q-1\\
  \end{array}\right),$
the  eigenvalues  of $B(A)$ are the roots of the equation
\begin{equation}\label{eq55}
\lambda^3 - (n - 3)\lambda^2 + (pq - 2n + 3)\lambda + pq - n + pqk + 1 = 0.\end{equation}
It is obvious that $\rho(A({K}(n,k,p)))$ is the largest root of the equation (\ref{eq55}).

Now we show $\rho(A({K}(n,k,1)))=\max\{\rho(A({K}(n,k,p)))| 1\leq p\leq n-k-1\}$.
We note that $p+q=n-k$ and the   adjacency matrix is symmetric, without loss of generality,
we assume that $q \geq p \geq 1$.
Let $f_{p,q}(\lambda)=\lambda^3 - (n - 3)\lambda^2 + (pq - 2n + 3)\lambda + pq - n + pqk + 1$. % be the left side of the above equation
Let $H= {K_k} \bigtriangledown ({K_{p-1}} \cup {K}_{q+1})={K}(n,k,p-1)$. Obviously, $H \in \mathcal C(n,k)$ and $\rho(H)$ is the largest root of  $f_{p-1,q+1}(\lambda) = 0$, then

$ f_{p,q}(\lambda) - f_{p-1,q+1}(\lambda)$
$=pq\lambda+pq+pqk-(p-1)(q+1)\lambda-(p-1)(q+1)-(p-1)(q+1)k$

\hskip3.65cm $=(q+1-p)(\lambda+k+1) > 0$,

\noindent and

$f_{p,q}(\rho(H)) = f_{p,q}(\rho(H)) - f_{p-1,q+1}(\rho(H))>0=f_{p,q}(\rho({K}(n,k,p))).$

%\hskip2cm $=pq\rho(H)+pq+pqk-(p-1)(q+1)\rho(H)-(p-1)(q+1)-(p-1)(q+1)k$

%\hskip2cm $=(q+1-p)(\rho(H)+k+1) > 0$.
%It implies  that  $f_{p,q}(\rho(H)) \geq 0 = f_{p,q}(\rho({K}(n,k,p)))$ for $\lambda \geq \rho(G)$,
 It implies $\rho(H)=\rho({K}(n,k,p-1)) > \rho({K}(n,k,p))$.
 Thus $\rho({G})\leq \rho(K(n,k,1))$,  $\rho({{K}(n,k,1)})$ is the largest root of the equation (\ref{eq51}),
 $\rho({G})=\rho(K(n,k,1))$ if and only if $G = K(n,k,1).$

Second, we show (ii) holds. We apply Lemma \ref{lem34} to $Q=Q({K}(n,k,p))$.
 Since $t=3$, $l_1=l_2=l_3=1,$  $p_1=p+k-2, $ $p_2=n-2,$ $p_3=n-p-2,$  $s_{13}=s_{31}=0$ and $s_{12}=s_{21}=s_{23}=s_{32}=1$, we have $\sigma(Q)=\sigma(B(Q))\cup\{(p+k-2)^{[p-1]},(n-2)^{[k-1]},(n-p-2)^{[q-1]}\}$ where the corresponding equitable  quotient matrix of $Q$ is
$$B(Q)=\left(\begin{array}{ccc}
 2p+k-2 &  k  & 0\\
 p  &  k+n-2 & q\\
 0  &  k & q+n-p-2\\
  \end{array}\right), $$
the eigenvalues of $B(Q)$ are $n-2, n-2+\frac{k}{2}\pm\frac{1}{2}\sqrt{(k-2n)^2+16p(k-n+p)}$.
Thus $\rho(Q)=n-2+\frac{k}{2}+\frac{1}{2}\sqrt{(k-2n)^2+16p(k-n+p)}.$

Let $f(p)=n-2+\frac{k}{2}+\frac{1}{2}\sqrt{(k-2n)^2+16p(k-n+p)}$, then $f(1)=f(n-k-1)=\max\{f(p) | 1\leq p\leq n-k-1\}.$
 Therefore, $q({G})\leq \frac{2n+k-4+\sqrt{(2n-k-4)^2+8k}}{2},$ and we complete the proof of (ii) by ${K}(n,k,1)\cong {K}(n,k,n-k-1)$.

Third, we show (iii) holds. We apply Lemma \ref{lem34} to $\mathcal{D}=\mathcal{D}({K}(n,k,p))$.
Since $t=3$, $l_1=l_2=l_3=1,$  $p_1=p_2=p_3=-1,$  $s_{13}=s_{31}=2$  and $s_{12}=s_{21}=s_{23}=s_{32}=1$, we have $\sigma(\mathcal{D})=\sigma(B(\mathcal{D}))\cup\{(-1)^{[n-3]}\}$ where the corresponding equitable quotient matrix of $\mathcal{D}$ is
$$B(\mathcal{D})=\left(\begin{array}{ccc}
 p-1 &  k  & 2q\\
 p  &  k-1 & q\\
 2p  &  k & q-1\\
  \end{array}\right), $$
the eigenvalues  of $B(\mathcal{D})$ are the roots of the equation:
\begin{equation}\label{eq56}
\lambda^3-(n-3)\lambda^2-(3pq+2n-3)\lambda+pqk-3pq-n+1=0.\end{equation}
It is obvious that $\rho(\mathcal{D}({K}(n,k,p)))$ is the largest root of the equation (\ref{eq56}).

Similar to the proof of (i), we can show (iii) holds, we omit it.

Finally,  we show (iv) holds. We apply Lemma \ref{lem34} to $\mathcal{Q}=\mathcal{Q}({K}(n,k,p))$.
 Since $t=3$, $l_1=l_2=l_3=1,$ $ p_1=n+q-2,$ $p_2=n-2,$ $p_3=n+p-2,$ $ s_{13}=s_{31}=2$ and $s_{12}=s_{21}=s_{23}=s_{32}=1$, we have $\sigma(\mathcal{Q})=\sigma(B(\mathcal{Q}))\cup\{(n+q-2)^{[p-1]},(n-2)^{[k-1]},(n+p-2)^{[q-1]}\} $
 where the corresponding equitable quotient matrix of $\mathcal{Q}$ is
$$B(\mathcal{Q})=\left(\begin{array}{lcr}
 n+p+q-2 &  k  & 2q\\
 p  &  n+k-2 & q\\
 2p  &  k & n+p+q-2\\
  \end{array}\right),$$
 the eigenvalues of $B(\mathcal{Q})$ are the roots of the equation:
$\lambda^3-(5p+5q+4k-6)\lambda^2+(8p^2+8q^2+5k^2+12pq+13pk+13qk
  -20p-20q-16k+12)\lambda-4p^3-4q^3-2k^3-8p^2q-8pq^2-10p^2k
  -10q^2k-8pk^2-8qk^2-16pqk+16p^2+16q^2+10k^2+24pq+26pk+26qk
  -20p-20q-16k+8=0.$

Similar to the proof of (i), we can show (iv) holds, we omit it.
\end{proof}

It is natural that  whether there exists similar result for
 the  Laplacian spectral radius or the  distance Laplacian spectral radius  in $\mathcal{C}(n,k)$ or not?
 In fact, we can obtain the spectrum of  the  Laplacian matrix or the  distance Laplacian matrix of $K(n,k,p)$ immediately.

\begin{prop}\label{prop52}
Let $K(n,k,p)$ defined as before. Then

(i). $\sigma(L(K(n,k,p)))=\{0, k, n^{[k]},(p+k)^{[p-1]}, (q+k)^{[q-1]}\}.$

(ii). $\sigma(\mathcal{L}(K(n,k,p)))=\{0, n+p+q, n^{[k]}, (n+q)^{p-1},(n+p)^{q-1}\}.$
\end{prop}
\begin{proof}
Firstly, the  Laplacian matrix $L(K(n,k,p))$ and
the  distance Laplacian matrix $\mathcal{L}(K(n,k,p))$ of $K(n,k,p)$ are the following matrices,
where $q=n-p-k$.

$$L=L(K(n,k,p)) = \left(\begin{array}{ccc}
  -J_p+(p+k)I_p &  -J_{p,k} & \mathbf{0}_{p,q}\\
  -J_{k,p} & -J_k+nI_k  &  -J_{k,q}\\
  \mathbf{0}_{q,p} &  -J_{q,k} & -J_q+(q+k)I_q\\
  \end{array}\right),$$
$$\mathcal{L}=\mathcal{L}(K(n,k,p)) = \left(\begin{array}{ccc}
 -J_p+(n+q)I_p &  -J_{p,k} &  -2J_{p,q}\\
  -J_{k,p} & -J_k+nI_k &  -J_{k,q}\\
  -2J_{q,p} &  -J_{q,k} &  -J_q+(n+p)I_q\\
  \end{array}\right).$$

Then the corresponding equitable quotient matrices are as follows:
$$B(L)= \left(\begin{array}{ccc}
 k &    -k  & 0\\
 -p  &    n-k  &  -q\\
 0   &   -k  &  k\\
  \end{array}\right), \qquad B(\mathcal{L})= \left(\begin{array}{ccc}
 n+q-p &    -k  & -2q\\
 -p  &    n-k  &  -q\\
 -2p   &   -k  &  n+p-q\\
  \end{array}\right).$$

  Then by Lemma \ref{lem34} and directly calculating, we obtain (i) and (ii).
\end{proof}

\end{document}